\documentclass[12pt]{amsart}
\usepackage{epsfig,amsfonts,amsthm,amssymb,latexsym,amsmath}
\usepackage[latin1]{inputenc}
\usepackage{bm}
\usepackage{cite}

\newfont{\bms}{msbm10 scaled 1100}
\def\matR{\mbox{\bms R}}
\def\matQ{\mbox{\bms Q}}
\def\matZ{\mbox{\bms Z}}
\def\matK{\mbox{\bms K}}

\def\matN{\mbox{\bms N}}

\newfont{\bmsi}{msbm10 scaled 700}

\def\mattQ{\mbox{\bmsi Q}}

\def\mattK{\mbox{\bmsi K}}
\def\mattL{\mbox{\bmsi L}}

\newfont{\bmsw}{msbm10 scaled 2400}

\newfont{\bmst}{msbm10 scaled 1600}

\theoremstyle{plain}
\newtheorem{theorem}{Theorem}[section]
\newtheorem{corollary}[theorem]{Corollary}
\newtheorem{remark}[theorem]{Remark}
\theoremstyle{definition}
\newtheorem{definition}[theorem]{Definition}
\theoremstyle{example}

\theoremstyle{proposition}
\newtheorem{proposition}[theorem]{Proposition}
\theoremstyle{lemma}
\newtheorem{lemma}[theorem]{Lemma}
\vspace{5pt}

\title{On the existence of rotated $D_n$-lattices constructed  via  Galois extensions}
\vspace{5pt}

\thanks{This work was partially
supported by CAPES 2548/2010, CNPq 150802/2012-9, 309561/2009-4 and
FAPESP 2007/56052-8}

\author{\scriptsize Grasiele C. \ Jorge  \ and \
\ Sueli I. R. \ Costa}
\date{}

\begin{document}
\maketitle

\vspace{-20pt}
\begin{center}
{\footnotesize  IMECC-UNICAMP ,
\\ 651, St. Sérgio Buarque de Holanda, 13083-859, Campinas, SP, BRAZIL \\
Emails: [grajorge, sueli] \ @ime.unicamp.br \\
}\end{center}

\hrule

\begin{abstract}
In this paper we construct families of rotated $D_n$-lattices which
may be suitable for signal transmission over both Gaussian and
Rayleigh fading channels via subfields of cyclotomic fields. These
constructions exhibit full diversity and good minimum product
distance which are important parameters related to the signal
transmission error probability. It is also shown that for some
Galois extensions $\matK|\matQ$ it is impossible to construct
rotated $D_n$-lattices via fractional ideals of $\mathcal
O_{\mattK}$.
\end{abstract}

\medskip
\noindent 

\medskip
\noindent 
\medskip

\hrule



\section{Introduction}

 A \textit{lattice} $\Lambda$ is a discrete additive subgroup of
$\mathbb{R}^n$. Equivalently, $\Lambda \subseteq \mathbb{R}^n$ is a
lattice iff there are linearly independent vectors ${\bm
v_1},\ldots,{\bm v_m} \in \mathbb{R}^n\mbox{, } m \leq n$, such that
any ${\bm y} \in \Lambda$ can be written as ${\bm y} = \sum_{i=1}^m
\alpha_i {\bm v_i} \mbox{,} \,\, \alpha_i \in \mathbb{Z}$. The set
$\{{\bm v_1},\ldots,{\bm v_m}\}$ is called a \textit{basis} for
$\Lambda$. A matrix $M$ whose rows are these vectors is said to be a
\textit{generator matrix} for $\Lambda$ while the associated
\textit{Gram matrix} is $G = M M^{t}$. The \textit{determinant} of
$\Lambda$ is $\det \Lambda = \det{G}$ and it is an invariant under
basis change \cite{sloane}.

Lattices are used in applications of different areas, particularly
in information theory and more recently in cryptography. In this
paper we attempt to construct lattices with full rank (m=n) which
may be suitable for signal transmission over both Gaussian and
Rayleigh fading channels \cite{boutros}. For this purpose the
lattice parameters that we consider here are packing density,
diversity and minimum product distance \cite{sloane,Oggier}.

The {\it packing density}, $\Delta(\Lambda)$, of a lattice $\Lambda$
is the proportion of the space $\matR^{n}$ covered by the union of
congruent disjoint spheres of maximum radius centered at the points
of $\Lambda$ \cite{sloane}. A lattice $\Lambda$ has {\it diversity}
$k\leq n$ if $k$ is the maximum number such that for all ${\bm
y}=(y_1,\cdots,y_n) \in \Lambda$, ${\bm y} \neq {\bm 0}$, there are
at least $k$ nonzero coordinates. Given a full rank lattice with
full diversity $\Lambda \subseteq \matR^{n}$ $(k=n)$ the {\it
minimum product distance} is defined as $d_{min}(\Lambda) =
\min\left\{\prod_{i=1}^{n}|y_i|,\,\,\, {\bm y}=(y_1,\cdots,y_n) \in
\Lambda,{\bm  y} \neq {\bm 0} \right\}$ \cite{Oggier}.

Special algebraic lattice constructions allows to derive certain
lattice parameters, which are usually difficult to be calculated for
general lattices in $\matR^{n}$, such as the packing density and the
minimum product distance. Some of these constructions  can be found
in \cite{boutros,andre,Oggier,gabriele,evaivan}.

In \cite{grasiele}  we have constructed families of full diversity
rotated $D_n$-lattices through of a twisted homomorphism and a free
$\matZ$-module  of rank $n$ contained in the totally real subfield
$\matQ(\zeta_m+\zeta_m^{-1})$ of the cyclotomic field
$\matQ(\zeta_m)$ for $n=\phi(m)/2$, where $\phi$ is the Euler
function. For $\matK_1=\matQ(\zeta_{2^r}+\zeta_{2^r}^{-1})$, $r \geq
5$, rotated $D_{2^{r-2}}$-lattices were obtained via free
$\matZ$-modules which are ideals of $\mathcal O_{\mattK_1}$ whereas
for $\matK_2=\matQ(\zeta_p+\zeta_p^{-1})$, $p \geq 7$ prime, rotated
$D_{(p-1)/2}$-lattices were constructed via free $\matZ$-modules $M
\subseteq \mathcal O_{\mattK_2}$ of rank $n = [\matK_2:\matQ]$ that
are not ideals of $\mathcal O_{\mattK_2}$.

In this paper, considering the compositum $\matK_3$ of
$\matQ(\zeta_{2^r}+\zeta_{2^r}^{-1})$ and
$\matQ(\zeta_p+\zeta_p^{-1})$ and the compositum $\matK_4$ of
$\matQ(\zeta_{p_1}+\zeta_{p_1}^{-1})$ and
$\matQ(\zeta_{p_2}+\zeta_{p_2}^{-1})$, where $p,p_1,p_2$ are prime
numbers, $p_1 \neq p_2$, we construct families of rotated
$D_n$-lattices via free $\matZ$-modules of rank $n$ that are not
ideals. In both cases we get a closed-form for their minimum product
distances. If it was possible to construct such rotated
$D_n$-lattices via principal ideals of $\mathcal O_{\mattK_i}$,
$i=2,3,4$, their minimum product distances should be twice the ones
obtained in our construction and therefore more appropriate to our
purpose. However, we show here that in the cases we have considered
it is impossible to construct rotated $D_n$-lattices via fractional
ideals of $\mathcal O_{\mattK_i}$, $i=2,3,4$. Moreover, we present a
necessary condition to construct these lattices via a Galois
extension $\matK|\matQ$. These results are related to an open
problem stated in \cite{eva2}: Given a lattice $\Lambda$, is there a
number field $\matK$ such that it is possible to construct $\Lambda$
via $\mathcal O_{\mattK}$?

This paper is organized as follows. In Section 2 we collect some
definitions and results on ideal lattices. In Section 3 new
constructions of rotated $D_n$-lattices obtained via the compositum
of some subfields of cyclotomic fields are derived. Section 4 is
devoted to a discussion on the existence of rotated $D_n$-lattices
constructed via fractional ideals of a Galois extension
$\matK|\matQ$.

\section{Ideal lattices}

The construction of ideal lattices presented here was introduced in
\cite{eva2}.

In what follows we use some classical notation for number fields
\cite{pierre,was,marcus}. Let $\matK$ be a Galois extension of
degree $n = [\matK:\matQ]$,
$Gal(\matK|\matQ)=\{\sigma_i\}_{i=1}^{n}$ the Galois group,
$d_{\mattK}$ the discriminant of $\matK|\matQ$, $\mathcal
O_{\mattK}$ its ring of integers, $N_{\mattK|\mattQ}(x)$ and
$Tr_{\mattK|\mattQ}(x)$ the norm and the trace of an element $x \in
\matK$, respectively.


\begin{definition} Let $\matK$ be a totally real number field and $\alpha \in \matK$ such that $\alpha_i=\sigma_i(\alpha)
> 0$ for all $i=1,\cdots,n.$ The homomorphism

$$\left.\begin{array}{l} \sigma_{\alpha}:\matK \longrightarrow
\matR^{n} \\
\hspace{0.8cm} x\longmapsto
\left(\sqrt{\alpha_{1}}\sigma_1(x),\ldots,\sqrt{\alpha_{n}}\sigma_{n}(x)\right)\end{array}
\right.$$ is called {\it twisted homomorphism}. When $\alpha=1$ the
twisted homomorphism is the so called {\it Minkowski homomorphism}.
\end{definition}

\begin{proposition}\cite{Oggier} If $\mathcal I\subseteq {\matK}$ is a free $\matZ$-module of rank
$n$ with $\matZ$-basis $\{w_1,\ldots,w_n\}$, then the image $\Lambda
= \sigma_{\alpha}(\mathcal I)$  is a lattice in $\matR^n$ with basis
\linebreak $\{{ \sigma_{\alpha}(w_1)},\ldots,{
\sigma_{\alpha}(w_n)}\}.$ Moreover, if $\matK$ is totally real,
$\Lambda$ is a full diversity lattice which has  $ G= M.M^{t} =
\left( Tr_{\mattK|\mattQ}(\alpha w_{i}{w_{j}}) \right)_{i,j=1}^{n}$
as a Gram matrix.
\end{proposition}

\begin{proposition}\cite{grasiele}\label{distanciaminima} Let $\matK$ be a totally real field number with  $[\matK:\matQ]=n$ and $\mathcal I \subseteq
{\matK}$  a free $\matZ$-module of rank $n$. The {\it minimum
product distance} of $\Lambda =\sigma_{\alpha}(\mathcal I)$ is $
d_{p,min}(\Lambda)=\sqrt{N_{\mattK|\mattQ}(\alpha)}min_{0\neq y\in
\mathcal I}|N_{\mattK|\mattQ}(y)|.$
\end{proposition}

\begin{proposition} \cite{Oggier} \label{distanciaminima1} If $\matK$ is a totally real field number, $\mathcal I \subseteq \mathcal O_{\mattK}$
is a principal ideal and $\Lambda=\sigma_{\alpha}(\mathcal I)$, then
$ d_{p,min}(\Lambda)= \sqrt{\frac{det(\Lambda)}{|d_{\mattK}|}}.$
\end{proposition}

\begin{definition} \label{relative} The {\it relative minimum product distance} of $\Lambda,$ denoted by ${\bm
d_{p,rel}(\Lambda)}$, is the minimum product distance of a scaled
version of $\Lambda$ with unitary minimum norm vector.
\end{definition}

Proposition \ref{detb} was presented in \cite{eva2} in the case in
which $\mathcal I \subseteq \matK$ is a fractional ideal. We can
extend it to free $\matZ$-modules of rank $n$.  Using the next lemma
the proof of Proposition \ref{detb} is straightforward.

\begin{lemma}\label{grasi}  If $\mathcal I\subseteq \matK$ is a free $\matZ$-module of rank $n$, then there
is  $d\in\matZ-\{0\}$ such that $d \mathcal I\subset \mathcal
O_{\mattK}$.\end{lemma} \begin{proof} Let $\alpha\in\matK$ such that
$\matK=\matQ(\alpha)$ and $\{1,\alpha,\cdots,\alpha^{n-1}\}$ is a
basis of $\matK$ over $\matQ$. Let $\{\gamma_1,\cdots, \gamma_n\}$
be a $\matZ$-basis of $\mathcal I$. For any $i$,
$\gamma_i=\sum_{j=0}^{n-1}a_{ij}\alpha^{j}$ where $a_{ij}\in\matQ$,
for all $i=1,\cdots,n$ and $j=0,1,\cdots,n-1$. Since
$a_{ij}\in\matQ$ for all $i,j$, it follows that
$a_{ij}=\frac{b_{ij}}{c_{ij}}$ with $b_{ij}, c_{ij} \in\matZ$ and
$c_{ij}\neq 0$ for all $i,j$. If $d=lcm\{c_{ij}; \,\,
i=1,\cdots,n,\, j=0,1,\cdots,n-1\}$, then $d\gamma_i \in
\matZ[\alpha]$ for all $i=1,\cdots,n$, what implies $d\mathcal I=d
\sum_{i=1}^{n}\matZ \gamma_i = \sum_{i=1}^{n}\matZ d \gamma_i
\subset \matZ [\alpha] \subset \mathcal O_{\mattK}$. \end{proof}

\begin{proposition}\label{detb} Let $\mathcal I\subseteq \matK$ be a nonzero free $\matZ$-module of rank $n$.
Then
\begin{equation}  \label{det} det(\sigma_{\alpha}(\mathcal I))= N(\mathcal I)^{2}
N_{\mattK|\mattQ}(\alpha)d_{\mattK}.\end{equation}
\end{proposition}

\section{Rotated $D_n$-lattices}

 For $n \geq 3$, the $n$-dimensional lattice $D_n$ in $\matR^{n}$ is
described, in its standard form, as  $D_n=\{(x_1,\cdots,x_n) \in
\matZ^{n}; \sum_{i=1}^{n}x_i\,\,\, \mbox { is even}\}$. The set
$\beta= \{(-1,-1,0,\cdots,0),$ $(1, -1,0, \cdots,0),$ \linebreak
$(0,1,-1,0,\cdots,0),\cdots,(0,0,\cdots,1,-1)\}$ is a basis for
$D_n$ and $det(D_n)=4$ for all $n$.

As it is known, $D_n$ has a very efficient decoding algorithm and
also has the best lattice packing density for $n=3,4,5$
\cite{sloane}.

Using subfields of some cyclotomic fields we have obtained in
\cite{grasiele} full diversity rotated $D_n$-lattices.

\begin{proposition}\cite{grasiele}\label{idealI} Let $\matK=\matQ(\zeta_{2^r}+\zeta_{2^r}^{-1})$, $e_0=1$, $e_i=\zeta_{2^r}^{i}+\zeta_{2^r}^{-i}$ and $n=2^{r-2}$. If
$\mathcal I \subseteq \mathcal O_{\mattK}$ is the $\matZ$-module
with $\matZ$-basis $\{-2e_0 + 2 e_1 - 2e_2 + \cdots -2e_{n-2} +
e_{n-1},-e_{n-1},e_{n-2},\cdots, (-1)^{i+1}e_{n-1-i}, \cdots,
e_2,-e_1\}$ and $\alpha=2+e_1,$ then the lattice
$\frac{1}{\sqrt{2^{r-1}}}\sigma_{\alpha}(\mathcal I) \subseteq
\matR^{2^{r-2}}$ is a rotated $D_n$-lattice and $\mathcal I = e_1
\mathcal O_{\mattK}$ is a principal ideal of $\mathcal O_{\mattK}$.
\end{proposition}

\begin{proposition} \cite{grasiele} \label{rotacionado} Let $\matK=\matQ(\zeta_p + \zeta_p^{-1})$, $e_i=\zeta_{p}^{i}+\zeta_{p}^{-i}$  and
$n=\frac{p-1}{2}$. If $\mathcal I \subseteq \mathcal O_{\mattK}$ is
the $\matZ$-module with $\matZ$-basis $\{-e_1 - 2e_2 - \cdots -2e_n,
e_1,e_2,\cdots,e_{n-1}\}$ and $\alpha=2-e_1$, then the lattice
$\frac{1}{\sqrt{p}}\sigma_{\alpha}(\mathcal I) \subseteq
\matR^{\frac{p-1}{2}}$ is a rotated $D_n$-lattice. The
$\matZ$-module $\mathcal I \subseteq \mathcal O_{\mattK}$ is not an
ideal of $\mathcal O_{\mattK}$. \end{proposition}

In what follows we consider the compositum of the fields
$\matQ(\zeta_{2^r}+\zeta_{2^r}^{-1})$ and $\matQ(\zeta_p +
\zeta_p^{-1})$, $p$ prime, to present new constructions of rotated
$D_n$-lattices.  These lattices are obtained via free
$\matZ$-modules of rank $n$ which are not ideals.

\begin{remark}\label{even} A remark to be used in the next proposition is that if $R$ is a rotation in $\matR^{n}$ and
${\bm x} \in R(\matZ^{n})$ such that $\|{\bm x}\|^{2}$ is even, then
${\bm x} \in R(D_n)$. In fact, if we start from  ${\bm y} \in
\matZ^{n}$ such that $||{\bm y}||^{2} = y_1^{2}+\cdots + y_n^{2} = 2
a$ for some $a \in \matN$, note that
$\left(\sum_{i=1}^{n}y_i\right)^{2} = \sum_{i=1}^{n}y_i^{2} + 2t =
2a + 2t$, $t \in \matZ$, is even. Then, $\sum_{i=1}^{n}y_i$ is even
and ${\bm y} \in D_n$.  Since a rotation $R$ preserves the norm of a
vector, given  ${\bm x}=R(\bm y),$ ${\bm y} \in \matZ^{n}$ such that
$||{\bm x}||^{2}$ is even, then ${\bm x} \in R(D_n)$. Particularly,
$D_n$ is the largest even sublattice of $\matZ^{n}$.
\end{remark}

\begin{proposition}\label{rotacionado11} Let $\matK_1=\matQ(\zeta_{2^r}+\zeta_{2^r}^{-1})$, $r \geq 3$, $e_0=1$,
$e_i=\zeta_{2^r}^{i}+\zeta_{2^r}^{-i}$, $i\geq 1$, and
$n_1=2^{r-2}$. Let $\matK_2=\matQ(\zeta_p + \zeta_p^{-1})$, $p \geq
5$ odd prime, $b_i=\zeta_{p}^{i}+\zeta_{p}^{-i}$, $i\geq 1$, and
$n_2=\frac{p-1}{2}$. If $\matK=\matK_1 \matK_2$ is the compositum of
$\matK_1$ and $\matK_2$, $\mathcal I \subseteq \mathcal O_{\mattK}$
is the $\matZ$-module with $\matZ$-basis $$\gamma=\{e_0b_1, e_0b_2,
\cdots, e_0b_{n_2-1}, 2e_0b_{n_2}, e_1b_1, \cdots, e_1b_{n_2},
\cdots, e_{n_1}b_1,\cdots, e_{n_1}b_{n_2}\},$$ $\alpha_1=(2-e_1)$,
$\alpha_2=(2-b_1)$, $\alpha=(2-e_1)(2-b_1)$, then the lattice
$\frac{1}{\sqrt{2^{r-1}p}}\sigma_{\alpha}(\mathcal I) \subseteq
\matR^{n}$ is a rotated $D_n$-lattice.
\end{proposition}  \begin{proof} From \cite{and} and \cite{Oggier} we have:

$Tr_{\mattK_1|\mattQ}(\alpha_1e_ie_i)=\left\{\begin{array}{l} 2n_1,\mbox{ if } i=0; \\
4n_1,\mbox{ if } i\neq 0.\end{array}\right.$  $
Tr_{\mattK_2|\mattQ}(\alpha_2 b_jb_j)=\left\{\begin{array}{l}
p,\mbox{ if } j=n_2; \\ 2p,\mbox{ if } j \neq
n_2.\end{array}\right.$

Using the fact that
$Tr_{\mattK|\mattQ}(\alpha_1e_ib_j\alpha_2e_kb_l) =
Tr_{\mattK_1|\mattQ}(\alpha_1e_ie_k)Tr_{\mattK_2|\mattQ}(\alpha_2
b_j b_l)$  we may conclude:

 $Tr_{\mattK|\mattQ}(\alpha e_i b_j e_i
b_j)=\left\{\begin{array}{cl} 2n_1 p, & \,\, \mbox{ if } i=0,j=n_2;
\\ 8 n_1 p,  & \,\, \mbox{ if } i\neq 0, j\neq n_2 \\ 4 n_1 p,  & \,\, \mbox{ otherwise}.
\end{array}\right.$

Considering the Gram matrix $G$ associated to the $\matZ$-basis
$\gamma$ we may assert that the value of the elements of its
diagonal are either $Tr_{\mattK|\mattQ}(\alpha 2 e_0 b_{n_2} 2 e_0
b_{n_2})$ \linebreak $=8 n_1 p$ or $Tr_{\mattK|\mattQ}(\alpha e_i
b_j e_i b_j) \in \{4 n_1 p,8 n_1 p\}$ for $i \neq 0$. From
\cite{Oggier}, we have that
$\Lambda_1=\frac{1}{\sqrt{2^{r-1}p}}\sigma_{\alpha}(\mathcal
O_{\mattK})$ is a rotated $\matZ^{n}$-lattice. Since $\mathcal I
\subseteq \mathcal O_{\mattK}$,  then
$\Lambda=\frac{1}{\sqrt{2^{r-1}p}}\sigma_{\alpha}(\mathcal I)
\subseteq \Lambda_1$. By Remark \ref{even}, $\sigma_{\alpha}(2e_0
b_{n_2} 2 e_0 b_{n_2}), \sigma_{\alpha}(e_i b_j e_i b_j) \in
R(D_n)$, where $R(D_n)$ is a rotated version of $D_n$ contained in
$\Lambda_1$. Then, $\Lambda \subseteq R(D_n)$. Moreover, $\Lambda =
R(D_n)$. In fact, since $d_{\mattK}=2^{(r-1)2^{r-2}-1}$
\cite{lopes}, by Proposition \ref{detb}, $
\det(\sigma_{\alpha}(\mathcal I)) = N(\mathcal I)^{2}
N_{\mattK|\mattQ}(\alpha)d_{\mattK} =
2^{2}(N_{\mattK_1|\mattQ}(\alpha_1))^{n_2}(N_{\mattK_2|\mattQ}(\alpha_2))^{n_1}
d_{\mattK_1}^{n_2}d_{\mattK_2}^{n_1}=$  $2^{2} 2^{n_2}p^{n_1}
2^{n_2((r-1)2^{r-2}-1)} p^{n_1\frac{p-3}{2}} = 4
 (2^{r-1}p)^{n_1n_2}.$  Therefore, $\Lambda =
R(D_n)$ is a rotated $D_n$-lattice. \end{proof}

\begin{proposition}\label{dist} The $\matZ$-module $\mathcal I \subseteq \mathcal O_{\mattK}$ given in Proposition \ref{rotacionado11} is not an ideal
of $\mathcal O_{\mattK}$. \end{proposition} \begin{proof} Indeed, if
$e_0 b_{n_2} \in \mathcal I$, then $\mathcal I = \mathcal
O_{\mattK}$. Since $\left|\frac{\mathcal O_{\mattK}}{\mathcal
I}\right| = 2,$ then $e_0b_{n_2} \not\in I.$  Note that $e_0
b_{n_2-1} e_0 b_{1} = e_0 b_{n_2} + e_0 b_{n_2-2}$ is not in
$\mathcal I$. In fact, if $e_0 b_{n_2-1} e_0 b_1 \in \mathcal I$,
then $e_0 b_{n_2} = e_0 b_{{n_2}-1} e_0 b_1 -e_0 b_{n_2-2} \in
\mathcal I$ and this does not happen. \end{proof}

\begin{proposition} If $\Lambda =\frac{1}{\sqrt{2^{r-1}p}}\sigma_{\alpha}(\mathcal I) \subseteq
\matR^{n}$ with $\alpha$ and $\mathcal I$ as in Proposition
\ref{rotacionado11}, then the relative minimum product distance is
${d_{p,rel}}(\Lambda)= 2^{\frac{-n r + n_2}{2}} p^{\frac{-n +
n_1}{2}}.$
\end{proposition} \begin{proof} From \cite{grasiele}, $|N_{\mattK_1|\mattQ}(b_1)| =1.$
Then, $|N_{\mattK|\mattQ}(b_1)|=
|(N_{\mattK_1|\mattQ}(b_1))^{n_2}N_{\mattK_2|\mattQ}(1)^{n_1}|$
\linebreak $=1.$ By Proposition \ref{distanciaminima},
${d_{p}(\sigma_{\alpha}(\mathcal I))}
=\sqrt{N_{\mattK_1|\mattQ}(\alpha)} min_{0\neq y \in \mathcal
I}|N_{\mattK_1|\mattQ}(y)| = \sqrt{2^{n_2}p^{n_1}},$ since
$min_{0\neq y \in \mathcal I}|N(y)| =1.$ The minimum Euclidean  norm
of $D_n$ is $\sqrt{2}$ what gives the expression for the relative
minimum product distance:

$d_{p,rel}(\Lambda) = \left(\frac{1}{\sqrt{2}}\right)^{n_1n_2}
\left(\frac{1}{\sqrt{2^{r-1}p}}\right)^{n_1n_2}
\sqrt{2^{n_2}p^{n_1}} = 2^{\frac{-n_1n_2r+n_2}{2}}
p^{\frac{-n_1n_2+n_1}{2}} .$
\end{proof}

In what follows we consider the compositum of the fields
$\matQ(\zeta_{p_1}+\zeta_{p_1}^{-1})$ and $\matQ(\zeta_{p_2} +
\zeta_{p_2}^{-1})$, $p_1,p_2$ distinct prime numbers, to construct
rotated $D_n$-lattices. In this construction we also obtain rotated
$D_n$-lattices via free $\matZ$-modules of rank $n$ which are not
ideals.

\begin{proposition}\label{rotacionado22} Let $\matK_1=\matQ(\zeta_{p_1}+\zeta_{p_1}^{-1})$,
$\matK_2=\matQ(\zeta_{p_2} + \zeta_{p_2}^{-1})$,  $p_1, p_2 \geq 5$
prime numbers, $p_1 \neq p_2$,
$e_i=\zeta_{p_1}^{i}+\zeta_{p_1}^{-i}$,
$b_i=\zeta_{p_2}^{i}+\zeta_{p_2}^{-i}$. If $\matK=\matK_1 \matK_2$
and $\mathcal I \subseteq \mathcal O_{\mattK}$ is the $\matZ$-module
with $\matZ$-basis $$\gamma_1=\{e_1b_1, e_1b_2, \cdots,
e_1b_{n_2-1}, e_1b_{n_2}, e_2b_1, \cdots, e_2b_{n_2}, \cdots,
e_{n_1}b_1,\cdots, 2e_{n_1}b_{n_2}\},$$  $\alpha_1=(2-e_1)$,
$\alpha_2=(2-b_1)$, $\alpha=(2-e_1)(2-b_1)$, then the lattice
$\frac{1}{\sqrt{p_1p_2}}\sigma_{\alpha}(\mathcal I) \subseteq
\matR^{n}$ is a rotated $D_n$-lattice.
\end{proposition} \begin{proof} From \cite{Oggier}: \\ $Tr_{\mattK|\mattQ}(\alpha e_i b_j e_i b_j)=\left\{\begin{array}{l}
2 p_1 p_2,\mbox{ if } i=n_1 \mbox{ and } j=n_2; \\ 2 p_1 p_2,\mbox{
if either } i \neq n_1 \mbox{ and } j = n_2 \mbox{ or } i = n_1
\mbox{ and } j \neq n_2 \\ 4 p_1 p_2,\mbox{ if } i \neq n_1 \mbox{
and } j \neq n_2
\end{array}\right.$ \\ Using the same ideas of Proposition \ref{rotacionado11}, we show that
$\frac{1}{\sqrt{p_1p_2}}\sigma_{\alpha}(\mathcal I) \subseteq
R(D_n)$. By Proposition \ref{detb} , $\det(\sigma_{\alpha}(I))  =
2^{2} p_1^{n_2}p_2^{n_1}
p_1^{\frac{p_1-1}{3}n_2}p_2^{\frac{p_2-1}{3}n_1} = 4
 (p_1p_2)^{n_1n_2}.$ Then, $\det(\Lambda)=4 = \det(R(D_n))$ and $\Lambda$ is a rotated $D_n$-lattice.
 \end{proof}

\begin{proposition} The $\matZ$-module  $I \subseteq \mathcal O_{\mattK}$ given in Proposition \ref{rotacionado22} is not an ideal of $\mathcal O_{\mattK}$. \end{proposition}
\begin{proof} Using the same argument of Proposition \ref{dist} we
show that \linebreak $e_{n_1} b_{n_2-1} e_{n_1} b_{1}  \not\in I$.
\end{proof}

\begin{proposition} If $\Lambda =\frac{1}{\sqrt{2^{r-1}p}}\sigma_{\alpha}(I) \subseteq
\matR^{n}$  as in the Proposition \ref{rotacionado22}, then the
relative minimum product distance is ${d_{p,rel}}(\Lambda)=
2^{\frac{-n_1n_2}{2}} p_1^{\frac{-n + n_2}{2}}p_2^{\frac{-n +
n_1}{2}}.$
\end{proposition} \begin{proof} From \cite{grasiele}, $|N_{\mattK_1|\mattQ}(b_1)| =1.$
So, $|N_{\mattK|\mattQ}(b_1)| =1.$ By Proposition
\ref{distanciaminima},  ${ d_{p}(\sigma_{\alpha}(I))}$
$=\sqrt{N_{\mattK_1|\mattQ}(\alpha)} min_{0\neq y \in
I}|N_{\mattK_1|\mattQ}(y)| = \sqrt{2^{n_2}p^{n_1}},$ since
$min_{0\neq y \in I}|N(y)|$ \linebreak $=1.$ The minimum Euclidean
norm of $D_n$ is $\sqrt{2}$ what gives the expression for the
relative minimum product distance:

$$d_{p,rel}(\Lambda) =
\left(\frac{1}{\sqrt{2}}\right)^{n_1n_2}
\left(\frac{1}{\sqrt{p_1p_2}}\right)^{n_1n_2}
\sqrt{p_1^{n_2}p_2^{n_1}} = 2^{\frac{-n_1n_2}{2}}
p_1^{\frac{-n_1n_2+n_2}{2}}p_2^{\frac{-n_1n_2+n_1}{2}}.$$

\end{proof}

For $\matK_1=\matQ(\zeta_p + \zeta_p^{-1})$, $p$ prime,
$\matK_2=\matQ(\zeta_p +
\zeta_p^{-1})\matQ(\zeta_{2^r}+\zeta_{2^r}^{-1})$, $p$ prime, and
$\matK_3=\matQ(\zeta_{p_1} +
\zeta_{p_1}^{-1})\matQ(\zeta_{p_2}+\zeta_{p_2}^{-1})$, $p_1, p_2$
primes, $p_1 \neq p_2$, if it was possible to construct these
rotated $D_n$-lattices via principal ideals of $\mathcal
O_{\mattK_i}$ we would have twice the relative minimum product
distances than the ones obtained in our constructions. This fact
motivated the discussion on the existence of such constructions via
ideals of $\mathcal O_{\mattK_i}$, $i=1,2,3$. It will be shown in
the next section that it is impossible to construct rotated
$D_n$-lattices via fractional ideals of $\mathcal O_{\mattK_i}$,
$i=1,2,3$.

\section{On the existence of rotated $D_n$-lattices constructed via Galois extensions}

From (\ref{det}), a necessary condition to construct a rotated
$D_n$-lattice, scaled by $\sqrt{c}$ with $c \in \matZ$, via
fractional ideals of $\mathcal O_{\mattK}$, is the existence of a
fractional ideal $\mathcal I \subseteq \mathcal O_{\mattK}$ and a
totally positive element $\alpha$ such that $ 4 c^{n} =
N_{\mattK|\mattQ}(\alpha) N(\mathcal I)^{2}d_{\mattK}.$

\begin{proposition}\cite{pierre} If $\mathcal P \subseteq \matZ$ is a nonzero prime ideal and $\matK|\matQ$ is a Galois extension, then
$\mathcal P\mathcal O_{\mattK}=\prod_{i=1}^{g}\mathcal Q_i^e,$ where
$\mathcal Q_{i\,'s}$ are distinct nonzero prime ideals of $\mathcal
O_{\mattK}.$ Moreover, $f(\mathcal Q_i|\mathcal P)=[\mathcal
O_{\mattL}/\mathcal Q_{i}:\mathcal O_{\mattK}/\mathcal P\mathcal
O_{\mattK}]= f$  and $n = e f g.$
\end{proposition}


\begin{proposition}\label{grasinha} Let  $\matK|\matQ$ be a Galois
extension, $[\matK:\matQ]=n$, $d_{\mattK}=2^{z} d$,  $z \in \matN$
and $d$ odd. If $2\mathcal O_{\mattK}=\prod_{i=1}^{g}{\mathcal
Q_i}^{e}$ and $f=f(\mathcal Q_i|2\mathcal O_{\mattK})$ does not
divide $2-z$, then it is impossible to construct rotated
$D_n$-lattices via fractional ideals of $\mathcal O_{\mattK}$.
\end{proposition} \begin{proof} Any ideal ${\mathcal B}$
of $\mathcal O_{\mattK}$ with even norm satisfies $N(\mathcal B) =
\left(2^{f}\right)^{a} b$ where $a \geq 1$ and $b$ is odd
\cite{pierre}. Let $\mathcal I$ be a fractional ideal of $\matK$. By
Lemma \ref{grasi}, there is $t \in \matZ$ such that $t \mathcal I =
\mathcal A$ is an integer ideal of $\mathcal O_{\mattK}$.  Then,
$N(\mathcal I) = N(\mathcal A)/t^{n}$. $N(\mathcal A)$ can be
written as $N(\mathcal A) =\left(2^{f}\right)^{a_1}b_1$ and so
$N(\mathcal I) = \frac{(2^{f})^{a_1} b_1}{t^{n}}$, with $a_1 \in
\matN$ and $b_1$ odd. Note that $N(\alpha) = N(\alpha \mathcal
O_{\mattK}) = \left(2^{f}\right)^{a_2} b_2$ with $a_2 \in \matN$,
$b_2$ odd. Let $t = 2^{k}l$, $k \in \matN$ and $l$ odd. Since $n= e
f g$, for all $c \in \matZ$ we have
\begin{equation}\begin{split} \label{fund} N(\mathcal I)^{2}
N(\alpha)d_{\mattK} & = \frac{(2^{f})^{2 a_1} {b_1}^{2}}{t^{2 n}}
(2^{f})^{a_2}b_2 2^{z} d =2^{2fa_1 - 2kn + fa_2 + z}
(b_1^{2}l^{-2n}b_2 d) \neq 4 c^{n}.
\end{split}\end{equation} Indeed, if $c
= 2^{a} b$ with $a \geq 0$ and $b$ odd, then $4 c^{n} = (2^{n})^{a}
2^{2} b^{n}$ and the powers of $2$ must be equal in (\ref{fund}),
therefore $a e f g +2 = an + 2 = f(2a_1+ a_2 - k e g )+ z$ and this
implies that $2 - z = f(2 a_1 + a_2 - a e g - k e g)$. By hypothesis
$f$ does not divide $2 -z$ and then it is impossible to find a
fractional ideal $\mathcal I$ and an element $\alpha \in \mathcal
O_{\mattK}$ satisfying the necessary conditions.   \end{proof}

\begin{remark} Let
$\matK=\matQ(\zeta_{2^r}+\zeta_{2^r}^{-1})$. In \cite{grasiele}, we
have presented a construction of  rotated $D_n$-lattices via ideals
of $\mathcal O_{\mattK}$. Note that $f | (2 - z)$. Indeed, we have
$2\mathcal O_{\mattK}=\mathcal P^{2^{r-2}}$ where  $\mathcal
P=(2-(\zeta_{2^r}+\zeta_{2^r}^{-1}))\mathcal O_{\mattK}$. Hence,
$f=1$ and $f$ divides $2-z$. A natural question to be answered is if
the condition $f|(2-z)$ is sufficient to allow the construction of
rotated $D_n$-lattices via fractional ideals, assuring a kind of
Proposition \ref{grasinha} reciprocal.
\end{remark}

We will show in Proposition \ref{gra1} that for any Galois extension
$\matK|\matQ$ with discriminant $d_{\mattK}$ odd it is impossible to
construct such rotated $D_n$-lattices via fractional ideals in
$\matK$. In particular, this assertion includes the fractional
ideals of $\matQ(\zeta_p+\zeta_p^{-1})$ and
$\matQ(\zeta_{p_1}+\zeta_{p_1}^{-1})\matQ(\zeta_{p_2}+\zeta_{p_2}^{-1})$,
$p, p_1, p_2$ prime numbers, $p\geq 7$, $p_1,p_2 \geq 5$, $p_1\neq
p_2$.

\begin{lemma} Let $\matK|\matQ$ be a Galois extension with odd discriminant.
Therefore,  $2 \mathcal O_{\mattK} = \prod_{i=1}^{s}\mathcal Q_i$
with $\mathcal Q_i \neq \mathcal Q_j$ prime ideals of $\mathcal
O_{\mattK}$.
\end{lemma} \begin{proof} Let $2 \mathcal O_{\mattK} =
\prod_{i=1}^{s}\mathcal Q_i^{e}$. From \cite{pierre}, $e
> 1$ if, and only if, $2\mathcal O_{\mattK}$ ramifies in $\mathcal
O_{\mattK}$. But this happens if, and only if, $2$ divides
$d_{\mattK}$, what does not happen. Therefore, $e=1$. \end{proof}

In the next lemma we use the Kummer's theorem to factorize ideals as
a product of prime ideals.

\begin{lemma} \label{lemmaa} If $\matK=\matQ(\theta)$ is a Galois extension such that $[\matK:\matQ] \not\in \{1,2,4\}$
and $2\mathcal O_{\mattK} = \prod_{i=1}^{s}\mathcal Q_i$, $\mathcal
Q_i \neq \mathcal Q_j$, then $f \neq 1$ and $f \neq 2$.  \end{lemma}
\begin{proof} Let $m(x) = min_{\mattQ}(\theta)$ and $\overline{m(x)}$ the
polynomial obtained from $m(x)$ by reduction modulo $\matZ_2[x]$. If
$f=1$,  $\overline{m(x)}$ is written as the product of irreducible
polynomials of degree  $1$ in $\matZ_2[x]$ and there are only two
possibilities for these polynomials, namely, $\overline{m_1(x)} =
\overline{x}$ and $\overline{m_2(x)} = \overline{x-1}$. If
$\overline{m(x)}=\overline{m_1(x)}\overline{m_2(x)}$, then $2
\mathcal O_{\mattK} = \mathcal Q_1 \mathcal Q_2$ and since $f=1$,
$N(\mathcal Q_1)=N(\mathcal Q_2) = 2^{1} = 2$. So it follows that
$N(2\mathcal O_{\mattK}) = 2^{2}$ and this is possible only if
$n=2$. If $\overline{m(x)}=\overline{m_1(x)}$, then $2 \mathcal
O_{\mattK} = \mathcal Q_1$ and $2^{n}=N(\mathcal Q_1)=2$ and this is
possible only if $n=1$. If $\overline{m(x)}=\overline{m_2(x)}$, we
also get $n=1$ in a similar way. If $f=2$, $\overline{m(x)}$ is
written as a product of irreducible polynomials of degree $2$ in
$\matZ_2[x]$ and there are only two possibilities for these
polynomials: $\overline{m_1(x)} = \overline{x^{2}}$ and
$\overline{m_2(x)} = \overline{x^2 + x + 1}$. If
$\overline{m(x)}=\overline{m_1(x)}\overline{m_2(x)}$, then $2
\mathcal O_{\mattK} = \mathcal Q_1 \mathcal Q_2$ and since $f=2$,
then $n=4$. If $\overline{m(x)}=\overline{m_1(x)}$ or
$\overline{m(x)}=\overline{m_2(x)}$ we also get $n=2$ in a similar
way. \end{proof}

\begin{proposition}\label{gra1} For any Galois extension $\matK|\matQ$ of degree $n \not\in \{1,2,4\}$ and $d_{\mattK}$ odd,
it is impossible to construct a rotated $D_n$-lattice via a twisted
homomorphism applied to a fractional ideal of $\mathcal O_{\mattK}$.
\end{proposition} \begin{proof} By Lemma \ref{lemmaa}, $f \neq \{1, 2\}$. Then $f$ does not divide $2$ and the statement follows
from Proposition \ref{grasinha}. \end{proof}

\begin{corollary}\label{gra2} If $\matK|\matQ$ is a Galois extension with conductor $m$ odd, then it is impossible to construct rotated $D_n$-lattices via
fractional ideals of $\mathcal O_{\mattK}$.
\end{corollary}
\begin{proof} This result follows immediately from the fact that if the
conductor $m$ is odd, $d_{\mattK}$ is also odd. \end{proof}

\begin{remark} By Corollary \ref{gra2}, it is impossible to construct rotated $D_3$ and $D_5$-lattices via fractional ideals of any Galois extension
$\matK \subseteq \matQ(\zeta_{m}+\zeta_{m}^{-1})$, $m$ odd.
\end{remark}

In what follows we will show that it is also impossible to construct
rotated $D_n$-lattices via fractional ideals of
$\matQ(\zeta_{2^{r}}+\zeta_{2^{r}}^{-1})\matQ(\zeta_{t}+\zeta_{t}^{-1})$,
$t$ odd.

The next lemma is presented in \cite{marcus} considering cyclotomic
fields $\matK_1=\matQ(\zeta_{2^{r}})$ and
$\matK_2=\matQ(\zeta_{t})$, $t$ odd. Using the same ideas we may
prove the analogous result for $\matQ(\zeta_{2^{r}} +
\zeta_{2^{r}}^{-1} )$ and $\matQ(\zeta_{t} + \zeta_{t}^{-1} )$.

\begin{lemma} \label{j3} Let  $t \in \matN$ odd, $\matK_1 = \matQ(\zeta_{2^{r}} + \zeta_{2^{r}}^{-1} )$, $\matK_2
= \matQ(\zeta_{t} + \zeta_{t}^{-1} )$, $[\matK_1: \matQ]=n_1$,
$[\matK_2: \matQ]=n_2$, $\matK=\matK_1 \matK_2$. If $2\mathcal
O_{\mattK_2}=\mathcal P_1\mathcal P_2\cdots \mathcal P_r$, $\mathcal
P_i\neq \mathcal P_j$ if $i\neq j$ where $\mathcal P_i$ are prime
ideals of $\mathcal O_{\mattK_2}$, then $2\mathcal
O_{\mattK}=(\mathcal Q_1\mathcal Q_2\cdots \mathcal Q_r)^{n_1},$
where $\mathcal Q_1,\cdots, \mathcal Q_r$ are prime ideals of
$\mathcal O_{\mattK}$ above of $\mathcal P_1,\cdots,\mathcal P_r,$
respectively and $f(\mathcal Q_i|\mathcal P)=f(\mathcal P_i|\mathcal
P)$, for all $i=1,\cdots,r.$
\end{lemma} \begin{proof} We have $2\mathcal O_{\mattK_1}
=\mathcal R^{n_1}$, where $\mathcal
R=(2-(\zeta_{2^k}+\zeta_{2^{k}}^{-1})\mathcal O_{\mattK_1}$ is a
prime ideal of $\mathcal O_{\mattK_1}$. Furthermore
 $2\mathcal O_{\mattK_2}=\mathcal P_1\mathcal P_2\cdots \mathcal
P_r$ a product of distinct prime ideals of $\mathcal O_{\mattK_2},$
with residual degree $f$. Consider $\mathcal Q_1,\cdots, \mathcal
Q_r$ prime ideals of $\mathcal O_{\mattK}$ above of $\mathcal P_1,
\cdots, \mathcal P_r$, respectively. We must have that $\mathcal
Q_i$ is above of $2 \mathcal O_{\mattK_1}$  and also above of
$\mathcal R$ for all $i=1,\cdots,r$. Let $2\mathcal O_{\mattK} =
(\mathcal Q_1\cdots \mathcal Q_r \mathcal Q_{r+1}\cdots \mathcal
Q_{s})^{\bar{e}}$, where $\mathcal Q_{i}^{'}s$ are prime ideals of
$\mathcal O_{\mattK}$ with residual degree $\bar{f}$. Then,
$s\bar{e}\bar{f}= n$. Note that $\bar{e}=e(\mathcal
Q_i|2)=e(\mathcal Q_i|\mathcal R)e(\mathcal R|2)=e(\mathcal
Q_i|\mathcal R)n_1$ and $\bar{f}=f(\mathcal Q_i|2)=f(\mathcal
Q_i|\mathcal P_i)f(\mathcal P_i|2)=f(\mathcal Q_i|\mathcal P_i)f.$
Then, $n =s\,\bar{e}\,\bar{f}=\,s\, e(\mathcal Q_i|\mathcal
R)\,n_1\, f(\mathcal Q_i|\mathcal P_i)\, f \geq r \,n_1\, f= n_1
n_2$. Therefore, $s\, e(\mathcal Q_i|\mathcal R)\,n_1\, f(\mathcal
Q_i|\mathcal P_i)\, f = r \,n_1\,f$, what implies that
$s\,e(\mathcal Q_i|\mathcal R)\, f(\mathcal Q_i|\mathcal P_i) = r$,
with $s\geq r$, $e(\mathcal Q_i|\mathcal R)\geq 1$, $f(\mathcal
Q_i|\mathcal P_i)\geq 1.$ So, $r=s$, $e(\mathcal Q_i|\mathcal
R)=f(\mathcal Q_i|\mathcal P_i)=1$, $\bar{e}=n_1$ and $\bar{f}=f$
and then $2\mathcal O_{\mattK}=(\mathcal Q_1\cdots \mathcal
Q_r)^{n_1}$ and $f=f(\mathcal Q_i|p)=f(\mathcal P_i|p).$ \end{proof}

\begin{proposition} Let
$\matK_1=\matQ(\zeta_{2^r}+\zeta_{2^r}^{-1})$,
$\matK_2=\matQ(\zeta_t + \zeta_t^{-1})$, $t$ odd,
$n_1=[\matK_1:\matQ]$ and $n_2=[\matK_2:\matQ] \not\in\{1,2,4\}$.
For $\matK=\matK_1 \matK_2$, the compositum of $\matK_1$ and
$\matK_2$, it is impossible to construct rotated $D_n$-lattices via
fractional ideals of $\mathcal O_{\mattK}$.
\end{proposition} \begin{proof} Since
$gcd(d_{\mattK_1},d_{\mattK_2})=1$, then $d_{\mattK}=
(d_{\mattK_1})^{n_2}(d_{\mattK_2})^{n_1} =$ \linebreak
$(2^{(r-1)2^{r-2}-1})^{n_2}(d_{\mattK_2})^{n_1}$ with $d_{\mattK_2}$
odd. Let $2\mathcal O_{\mattK}=\prod_{i=1}^{g} \mathcal Q_i^{e}$. We
will show that $f (2-z)=2-({(r-1)2^{r-2}-1})n_2$. Indeed,  $ n_1n_2
= e f g = n_1 f g$ by Proposition \ref{j3} and hence, $n_2 = f g$.
If $f|(2-z)$, then there is $q \in \matZ$ such that $f q = 2-z =
2-({(r-1)2^{r-2}-1})n_2= 2-({(r-1)2^{r-2}-1})f g$ and $f(q +
({(r-1)2^{r-2}-1})g)=2$. By hypothesis, $n_2 \not\in\{1,2,4\}$ and
by Proposition \ref{lemmaa}, $f\neq 1$ and $f\neq 2$. Therefore, $f$
does not divide $(2-z)$. \end{proof}

\begin{remark} In Table 1 we compare the relative minimum
product distances of rotated $D_n$-lattices obtained in Propositions
\ref{idealI}, \ref{rotacionado}, \ref{rotacionado11} and
\ref{rotacionado22} considering the number fields
$\matK_1=\matQ(\zeta_p+\zeta_p^{-1})$,
$\matK_2=\matQ(\zeta_{2^{r}}+\zeta_{2^r}^{-1})$,
$\matK_3=\matQ(\zeta_{2^{r_1}}+\zeta_{2^{r_1}}^{-1})\matQ(\zeta_{p_1}+\zeta_{p_1}^{-1})$
and
$\matK_4=\matQ(\zeta_{p_2}+\zeta_{p_2}^{-1})\matQ(\zeta_{p_3}+\zeta_{p_3}^{-1})$.

\tiny \begin{table}[h]
\begin{center}
\begin{tabular}{|c|c|c|c|c|c|c|c|c|c|c|}
 \hline
 $n$ & $p$ & $r$ & $r_1$ & $p_1$ & $p_2$ & $p_3$ & $\mattK_1$ & $\mattK_2$ &
 $\mattK_3$ & $\mattK_4$ \\
\hline
$3$ & $7$ & $-$ & $-$ & $-$ & $ - $ & $-$ & $0.369646$ & $-$ & $-$ & $-$\\
\hline
$4$ & $-$ & $4$ & $3$ & $5$ & $-$ & $-$ & $-$ & $0.324210$ & $0.281171$ & $-$\\
\hline
$5$ & $11$ & $-$ & $-$ & $-$ & $-$ & $-$ & $0.27097$ & $-$ & $-$ & $-$\\
\hline
$6$ & $13$ & $-$ & $3$ & $7$ & $-$ & $-$ & $0.24285$ & $-$ & $0.219793$ & $-$\\
\hline
$8$ & $17$ & $5$ & $4$ & $5$ & $-$ & $-$ & $0.20472$ & $0.201311$ & $0.182317$ & $-$\\
\hline $10$ & $-$ & $-$ & $3$ & $11$ & $-$ & $-$ & $-$ & $-$ &
$0.161122$ & $-$\\ \hline
$11$ & $23$ & $-$ & $-$ & $-$ & $-$ & $-$ & $0.17003$ & $-$ & $-$ & $-$\\
 \hline
$12$ & $-$ & $-$ & $3$ & $7$ & $-$ & $-$ & $-$ & $-$ & $0.144401$ & $-$\\
\hline
$14$ & $29$ & $-$ & $-$ & $-$ & $-$ & $-$ & $0.148086$ & $-$ & $-$ & $-$ \\
\hline
$15$ & $31$ & $-$ & $-$ & $-$ & $7$ & $11$ & $0.142402$ & $-$ & $-$ & $0.1380198$ \\
\hline
$16$ & $-$ & $6$ & $5$ & $5$ & $-$ & $-$ & $-$ & $0.133393$  & $0.123452$ & $-$ \\
\hline
$18$ & $37$ & $-$ & $3$ & $19$ & $-$ & $-$ & $0.128512$ & $-$ & $0.1136$ & $-$ \\
\hline
$20$ & $41$ & $-$ & $4$ & $11$ & $-$ & $-$ & $0.121175$ & $-$ & $0.104475$ & $-$ \\
\hline
$128$ & $257$ & $9$ & $-$ & $-$ & $-$ & $-$ & $0.0450746$ & $0.044554$ & $-$ & $-$ \\
\hline
$32768$ & $65537$ & $17$ & $-$ & $-$ & $-$ & $-$ & $0.00276258$ & $0.00276222$ & $-$ & $-$ \\
\hline
\end{tabular}
\caption{Relative minimum product distance of rotated $D_n$-lattices
associated to Galois extensions}
\end{center}
\end{table}
\end{remark}

In dimensions $8$, $128$ and $32768$ we have considered the Fermat
prime numbers $17$, $257$ and $65537$ to construct rotated
$D_n$-lattices via $\matZ$-modules in $\mathcal O_{\mattK_1}$ and
via ideals in $\mathcal O_{\mattK_2}$. It is worth noticing that, in
these cases, the relative minimum product distances obtained via
$\matZ$-modules in $\mathcal O_{\mattK_1}$ are greater than the ones
obtained via ideals in  $\mathcal O_{\mattK_2}$.

Note also that when both constructions are possible the relative
minimum product distances of rotated $D_n$-lattices obtained via
compositum of fields is smaller than the ones obtained via $\matK_1$
and $\matK_2$. On the other hand, rotated $D_n$-lattices constructed
via the compositum of fields can be obtained in many dimensions that
were not considered before.

\section{Conclusion}

In this paper we have obtained families of full diversity rotated
$D_n$-lattices via $\matZ$-modules in some compositum of Galois
fields. It is also shown that for any Galois extension with odd
discriminant it is impossible to construct rotated $D_n$-lattices
via fractional ideals of its ring of integers and established a
necessary condition to construct rotated $D_n$-lattices in a Galois
extension with even discriminant. As a consequence we can assert
that rotated $D_n$-lattices can not be constructed via fractional
ideals in Galois extensions with conductor $m$ odd.

\end{document}